\documentclass[12pt, letterpaper]{amsart}

\usepackage{amsmath,amssymb,latexsym, color
}
\usepackage[dvips,centering,includehead,width=15.1cm,height=22.7cm,
paperheight=23.8cm,paperwidth=17cm
]{geometry}
\usepackage{enumitem}
\usepackage{hyperref}
\usepackage{pict2e}
\usepackage{epic,graphicx}
\usepackage{rotating}
\usepackage{tikz-cd}
\usepackage{ccaption}
\changecaptionwidth
\captionwidth{5.6in}

\renewcommand{\ref}{\hyperref}

\def\Z{{\mathbb Z}}
\def\R{{\mathbb R}}
\def\Q{{\mathbb Q}}
\def\N{{\mathbb N}}\def\T{{\mathbb T}}
\def\C{{\mathbb C}}\def\V{{\mathbb V}}

\def\PP{{\mathbb P}}

\newcommand{\bi}{\boldsymbol{i}}

\newcommand{\bu}{\boldsymbol{u}}
\newcommand{\bz}{\boldsymbol{z}}
\newcommand{\bx}{\boldsymbol{x}}\newcommand{\bP}{\mathbf{P}}

\newcommand{\bn}{\boldsymbol{n}}

\DeclareMathOperator{\Aut}{Aut}

\DeclareMathOperator{\conv}{conv}

\DeclareMathOperator{\ord}{ord}

\DeclareMathOperator{\Sing}{Sing}

\newtheorem{theorem}{Theorem}[section]
\newtheorem{proposition}[theorem]{Proposition}

\newtheorem{lemma}[theorem]{Lemma}

\theoremstyle{definition}

\newtheorem{remark}[theorem]{Remark}

\numberwithin{equation}{section}

\newcommand{\eps}{{\varepsilon}}

\begin{document}

\title{Tropical vertex and real enumerative geometry}

\author{Eugenii Shustin}
\address{School of Mathematical Sciences, Tel Aviv
University, Tel Aviv 6997801,
Israel}
\email{shustin@tauex.tau.ac.il}


\thanks{The author was supported by the Bauer-Neuman Chair in Real and Complex Geometry}

\date{
\today
}

\begin{abstract}
We show that the commutator relations in the refined tropical vertex group can be expressed via the enumeration of suitable real rational curves in toric surfaces.
\end{abstract}



\maketitle





\section{Introduction}\label{secrtv1}

The tropical vertex group $\V$ is a subgroup of the group $\Aut_{\C[[t]]}\C[x,x^{-1},y,y^{-1}][[t]]$ of automorphisms of the $\C[[t]]$-algebra $\C[x,x^{-1},y,y^{-1}][[t]]$ (see, for instance, \cite[Sections 0.1, 0.2]{GPS}). It is
the completion with respect to the ideal $\langle t\rangle$ of the subgroup generated by all elements $\theta_{(a,b),f}$ such that
$$(a,b)\in\Z^2\setminus\{0\},\quad f=1+tx^ay^b\cdot g(x^ay^b,t),\ g(z,t)\in\C[z][[t]],$$
$$\theta_{(a,b),f}(x)=f^{-b}\cdot x,\quad \theta_{(a,b),f}(y)=f^a\cdot y.$$
The tropical vertex group appears in various settings, e.g. in the study of affine
structures, mirror symmetry, tropical geometry, in wall-crossing formulas for Donaldson-Thomas invariants, see an overview in \cite[Sections 0.2 and 1.1]{GPS} and \cite[Section 1]{FS}. The elements of $\V$ can be regarded as automorphisms of the formal one-parameter family of tori $(\C^*)^2$ which preserve the standard symplectic form $(xy)^{-1}dx\wedge dy$ in the fibers.

The main result of \cite{GPS} is an expression of the commutators in $\V$ via Gromov-Witten invariants
of certain toric surfaces blown-up at generic points on the toric divisors. Furthermore, it is shown that these Gromov-Witten invariants develop in sums of the relative Gromov-Witten invariants
$N^\C_{\alpha,\alpha'}(\bP_a,\bP_b)$ that count complex rational curves in the original toric surfaces with prescribed tangency conditions on toric divisors. Moreover, $N^\C_{\alpha,\alpha'}(\bP_a,\bP_b)=N^{trop}_{\alpha,\alpha'}(\bP_a,\bP_b)$, where the latter number enumerates appropriate plane rational tropical curves matching certain relative constraints and counted with Mikhalkin's weights \cite{Mi} (see also \cite{GM} and \cite{NS}).
S. A. Filippini and J. Stoppa \cite{FS} extended this result to the deformed tropical vertex group
$$\widehat\V\subset\Aut_{\C[[t]]}\C[\widehat x,\widehat x^{-1},\widehat y,\widehat y^{-1}][[t]],\quad\text{where}\ \widehat x\widehat y=q\widehat y\widehat x.$$ Namely, they showed that a similar formula for the commutators holds when substituting the relative refined tropical Block-G\"ottsche invariants $\widehat N^{trop}_{\alpha,\alpha'}(\bP_a,\bP_b)(q)$ (cf. \cite{BG}) for the numbers $N^{trop}_{\alpha,\alpha'}(\bP_a,\bP_b)$.
For $q=1$, the Filippini-Stoppa formula specializes to the Gross-Pandharipande-Siebert formula.

An interpretation of the $q$-refined tropical vertex was elaborated by Bousseau \cite{Bou}; it was based on the integration of pull-backs by evaluation maps combined with appropriate $\lambda$-classes along the virtual fundamental classes in the moduli spaces of marked complex curves of higher genus as developed in \cite{Bou0}.
Our approach is rather different: we relate the $q$-refined tropical vertex to enumeration of real rational curves in toric surface in the spirit of \cite{Mi1}.
It is worth to notice that Arg\"uz and Bousseau \cite[Section 1.3]{AB} proposed to study the tropical vertex from the real point of view. The {\it results} of this note are as follows:
\begin{enumerate}\item[(A)] the refined tropical invariants in the Filippini-Stoppa formula can be recovered from the enumeration of oriented real rational curves on toric surfaces with a given Mikhalkin's quantum index \cite{Mi1}, see Proposition \ref{prtv2};
\item[(B)] the specialization at $q=-1$ yields a series of enumerative invariants counting real rational curves on toric surfaces matching suitable constraints relative to the toric divisors, see Propositions \ref{prtv1} and \ref{lrtv3}.
    \end{enumerate}
    We point out that the above real enumerative invariants are of different nature. The former invariant counts embedded holomorphic Riemann surfaces with boundary on the real part of the ambient toric surface, it is related to the Block-G\"ottsche invariant multiplied by a certain factor (precisely as in \cite{Bou0,Bou} and \cite{Mi}). In turn, the latter invariant
    counts real curves and it is related to the evaluation of the genuine Block-G\"ottsche invariant at $q=-1$.

We set preparation details in Section \ref{secrtv3} and formulate precisely our results in Sections \ref{secrtv7} and \ref{secrtv6}, respectively. In Section \ref{app} we prove the invariance of the count of real rational curves matching constraints relative to toric divisors which is discussed in Section \ref{secrtv6}.

\section{Tropical vertex and its deformation}\label{secrtv3}

\subsection{Commutators in the tropical vertex group}\label{secrtv4}
Denote
$$S_{\ell_1}=\theta_{(1,0),(1+tx)^{\ell_1}},\quad T_{\ell_2}=\theta_{(0,1),(1+ty)^{\ell_2}},\quad
\ell_1,\ell_2\in\Z,\ \ell_1,\ell_2>0.$$
By \cite[Theorem 0.1]{GPS},
\begin{equation}T_{\ell_2}^{-1}\circ S_{\ell_1}\circ T_{\ell_2}\circ S_{\ell_1}^{-1}=\vec\prod_{(a,b)}\theta_{(a,b),f_{a,b}},\label{ertv3}\end{equation}
where the product means the composition over all primitive vectors $(a,b)\in\N^2\subset\Z^2$ ordered by the increasing slopes,
\begin{equation}\log f_{a,b}=\sum_{k\ge1}c^k_{a,b}(\ell_1,\ell_2)\cdot(tx)^{ak}(ty)^{bk},\label{ertv4}\end{equation}
and the coefficients $c^k_{a,b}(\ell_1,\ell_2)$ admit the following presentation:
\begin{equation}c^k_{a,b}(\ell_1,\ell_2)=\sum_{\renewcommand{\arraystretch}{0.6}
\begin{array}{c}
\scriptstyle{|\bP_{ka}|=ka}\\
\scriptstyle{\#\bP_{ka}=\ell_1}\end{array}}\sum_{\renewcommand{\arraystretch}{0.6}
\begin{array}{c}
\scriptstyle{|\bP_{kb}|=kb}\\
\scriptstyle{\#\bP_{kb}=\ell_2}\end{array}}N^\C_{a,b}(\bP_{ka},\bP_{kb}).\label{ertv5}\end{equation}
In the latter formula, $\bP_{ka}$ (resp., $\bP_{kb}$) ranges over all ordered partitions
$$p_1+...+p_{\ell_1}=ka,\ p_i\ge0,\quad (\text{resp.,}\ p'_1+...+p'_{\ell_2}=kb,\ p'_j\ge0),$$
and $N^\C_{a,b}(\bP_{ka},\bP_{kb})\in\Q$ is the Gromov-Witten invariant defined for the following moduli space $\overline{\mathcal M}_0(X_{a,b}^{\ell_1,\ell_2}/D_{out},\gamma)$ of virtual dimension zero (see all details in \cite[Section 4]{GPS}):
\begin{itemize} \item $X_{a,b}$ is the toric surface associated with the lattice triangle $\qquad$ $\qquad$ $\qquad$
$T_{a,b}=\conv\{(0,0),(a,0),(0,b)\}$, and $\pi:X_{a,b}^{\ell_1,\ell_2}\to X_{a,b}$ is the blow-up of $\ell_1$ generic points on the toric divisor $D_a$ associated with the edge $[(0,0),(a,0)]$ and of $\ell_2$ generic points on the toric divisor $D_b$ associated with the edge $[(0,0),(0,b)]$;
\item $\gamma\in H_2(X_{a,b}^{\ell_1,\ell_2})$ is the class represented by the divisor $\qquad$ $\qquad$ $\qquad$ $\qquad$ \mbox{$k\pi^*\Delta-\sum_{i=1}^{\ell_1}p_iE_i-\sum_{j=1}^{\ell_2}p'_jE'_j$} with $\Delta$ being the zero-divisor of a generic polynomial $\sum_{(i,j)\in T_{a,b}\cap\Z^2}a_{ij}x^iy^j$, and $E_1,...,E_{\ell_1},E'_1,...,E'_{\ell_2}$ the exceptional divisors of the blow-ups;
     \item $\overline{\mathcal M}_0(X_{a,b}^{\ell_1,\ell_2}/D_{out},\gamma)$ is the compactified moduli space of stable genus zero maps with a full contact order $k$ at an unspecified point of the toric divisor $D_{out}$ associated with the edge $[(a,0),(0,b)]$. \end{itemize}
Furthermore (see \cite[Theorem 2.8]{GPS}), the numbers $N^\C_{a,b}(\bP_{ka},\bP_{kb})$ can be expressed as
\begin{equation}N^\C_{a,b}(\bP_{ka},\bP_{kb})=\sum_{\alpha}\sum_{\alpha'}c(\bP_{ka},\bP_{bk},\alpha,\alpha')
N^\C_{\alpha,\alpha'}(\bP_a,\bP_b),\label{ertv2}\end{equation}
where
\begin{itemize}\item $\alpha$ (resp., $\alpha'$) ranges over collections of ordered partitions of each positive summand in $\bP_{ka}$ (resp., in $\bP_{kb}$) into positive summands:
$$p_i=\sum_j\alpha_{ij}\quad\forall p_i>0,\quad p'_i=\sum_j\alpha'_{ij}\quad\forall p'_i>0,$$ \item $c(\bP_{ka},\bP_{bk},\alpha,\alpha')\in\Q$ are coefficients depending only on $\bP_{ka},\bP_{kb},\alpha,\alpha'$,
\item $N^\C_{\alpha,\alpha'}(\bP_a,\bP_b)$ counts complex rational curves $\bn:\PP^1\to X_{a,b}$ realizing the divisor class $k\Delta$ and meeting the divisor $D_a$ at $\#\alpha$ ordered generic fixed points with multiplicities $\{\alpha_{ij}\}$, meeting the divisor $D_b$ at $\#\alpha'$ ordered generic fixed points with multiplicities $\{\alpha'_{ij}\}$, meeting the divisor $D_{out}$ at one point with multiplicity $k$ so that
    $$\bn^*D_a=\sum_{i,j}\alpha_{ij}w_{ij},\quad\bn^*D_b=\sum_{ij}\alpha'_{ij}w'_{ij},\quad\bn^*D_0=kw_0$$
    with $w_{ij}$'s, $w'_{ij}$'s, and $w_0$ being distinct points on $\PP^1$.
\end{itemize}

At last, the numbers $N^\C_{\alpha,\alpha'}(\bP_a,\bP_b)$ are equal to the following tropical enumerative invariants. Fix a generic sequence of $\#\alpha$ vertical lines $\{L(\alpha_{ij}\}_{i,j}\subset\R^2$ numbered by the elements of $\alpha$ and a generic sequence of $\#\alpha'$ horizontal lines $\{L(\alpha'_{ij})\}_{i,j}\subset\R^2$ numbered by the elements of $\alpha'$. Denote by ${\mathcal T}_0(\alpha,\alpha',\{L(\alpha_{ij})\}_{i,j},\{L(\alpha'_{ij})\}_{i,j})$ the set of rational tropical curves with $1+\#\alpha+\#\alpha'$ ends such that one end of weight $k$ is directed by the vector $(b,a)$, $\#\alpha$ (resp., $\#\alpha'$) ordered ends are directed by the vectors $(0,-1)$ (resp., $(-1,0)$), are equipped with the weights $\alpha_{ij}$ (resp., $\alpha'_{ij}$) and lie on the corresponding line $L(\alpha_{ij})$ (resp., $L(\alpha'_{ij})$).
By \cite[Proposition 4.13]{Mi}, the set ${\mathcal T}_0(\alpha,\alpha',\{L(\alpha_{ij})\}_{i,j},\{L(\alpha'_{ij})\}_{i,j})$ is finite and consists of only trivalent curves. Then (see \cite[Theorem 3.4]{GPS})
\begin{equation}N^\C_{\alpha,\alpha'}(\bP_a,\bP_b)=N^{trop}_{\alpha,\alpha'}(\bP_{ka},\bP_{kb}):=\sum_\Gamma\mu(\Gamma),
\label{ertv6}\end{equation}
where $\Gamma$ ranges over the set ${\mathcal T}_0(\alpha,\alpha',\{L(\alpha_{ij})\}_{i,j},\{L(\alpha'_{ij})\}_{i,j})$ and
$$\mu(\Gamma)=\left(\prod_{i,j}\alpha_{ij}\right)^{-1}\left(\prod_{i,j}\alpha'_{ij}\right)^{-1}
            \prod_{V\in\Gamma^0}\mu(\Gamma,V),$$ where $\Gamma^0$ is the set of (trivalent) vertices of $\Gamma$, $\mu(\Gamma,V)$ is the Mikhalkin multiplicity of the vertex $V$\ \footnote{The number $N^{trop}_{\alpha,\alpha'}(\bP_{ka},\bP_{kb})$ does not depend on the choice of a generic constraint
            $\{L(\alpha_{ij})\}_{i,j},\{L(\alpha'_{ij})\}_{i,j}$, see \cite{Mi} and \cite{GM}}.
            Thus, formula (\ref{ertv5}) turns into
            \begin{equation}c^k_{a,b}(\ell_1,\ell_2)=\sum_{\renewcommand{\arraystretch}{0.6}
\begin{array}{c}
\scriptstyle{|\bP_{ka}|=ka}\\
\scriptstyle{\#\bP_{ka}=\ell_1}\end{array}}\sum_{\renewcommand{\arraystretch}{0.6}
\begin{array}{c}
\scriptstyle{|\bP_{kb}|=kb}\\
\scriptstyle{\#\bP_{kb}=\ell_2}\end{array}}\sum_{\alpha,\alpha'}c(\bP_{ka},\bP_{bk},\alpha,\alpha')
N^{trop}_{\alpha,\alpha'}(\bP_{ka},\bP_{kb})
            \label{ertv9}\end{equation}

\subsection{Commutators in the deformed tropical vertex group}\label{secrtv5}
S. A. Filippini and J. Stoppa \cite{FS} considered the following deformation of the algebra $\C[x,x^{-1},y,y^{-1}][[t]]$:
\begin{equation}\C[\widehat x,\widehat x^{-1},\widehat y,\widehat y^{-1}][[t]],\quad\text{where}\ \widehat x\widehat y=q\widehat y\widehat x,\label{ertv8}\end{equation} and showed that the commutator
$$\widehat T_{\ell_2}^{-1}\circ\widehat S_{\ell_1}\circ\widehat T_{\ell_2}\circ\widehat S_{\ell_1}^{-1}$$
of the automorphisms
$$\widehat S_{\ell_1}=\theta_{(1,0),(1+q^{1/2}t\widehat x)^{\ell_1}},\quad \widehat T_{\ell_2}=\theta_{(0,1),(1+q^{1/2}t\widehat y)^{\ell_2}},\quad\ell_1,\ell_2>0,$$
admits the development similar to (\ref{ertv3}), (\ref{ertv4}), (\ref{ertv5}), and (\ref{ertv9}), in which the numerical tropical invariants $N^{trop}_{\alpha,\alpha'}(\bP_{ka},\bP_{kb})$ should be replaced with the relative Block-G\"ottsche tropical invariants as in \cite[Definition 7.2]{BG} 
\begin{equation}\widehat N^{trop}_{\alpha,\alpha'}(\bP_{ka},\bP_{kb})(q):=\sum_{\Gamma}\widehat\mu(\Gamma)(q),\label{ertv7}\end{equation}
where $\Gamma$ ranges over the set ${\mathcal T}_0(\alpha,\alpha',\{L(\alpha_{ij})\}_{i,j},\{L(\alpha'_{ij})\}_{i,j})$ and
\begin{equation}\widehat\mu(\Gamma)(q)=\left(\prod_{i,j}[\alpha_{ij}]_q\right)^{-1}\left(\prod_{i,j}
[\alpha'_{ij}]_q\right)^{-1}
            \prod_{V\in\Gamma^0}[\mu(\Gamma,V)]_q,\label{ertv11}\end{equation} where
            $$[\theta]_q=\frac{q^{\theta/2}-q^{-\theta/2}}{q^{1/2}-q^{-1/2}},\qquad\forall\ \theta\in\Z.$$
            Observe that $[\theta]_1=\theta$, and hence the Filippini-Stoppa formulas specializes at $q=1$ to the Gross-Pandharipande-Siebert formulas (\ref{ertv3}), (\ref{ertv4}), (\ref{ertv5}), and (\ref{ertv9}).

            Note also that the invariance of the expression (\ref{ertv7}) with respect to the choice of the constraint $\{L(\alpha_{ij})\}_{i,j},\{L(\alpha'_{ij})\}_{i,j}$ follows from \cite[Theorem 1]{IM}.

\section{Relative Block-G\"ottsche invariants and Mikhalkin's quantization}\label{secrtv7}

Here we show that the refined invariants $\widehat N^{trop}_{\alpha,\alpha'}(\bP_{ka},\bP_{kb})(q)$ can be recovered from enumeration of appropriate real rational curves.

We start with the following preparatory statement. Let $\beta=\{\beta_i\}_{i=1,...m}$, (resp., $\beta'=\{\beta'_i\}_{i=1,...,m'}$) be an ordered partition of $ka$ (resp., $kb$), and let $\bz=\{z_i\}_{i=1,...,m}$ (resp., $\bz'=\{z'_i\}_{i=1,...,m'}$) be a sequence of distinct points on the toric divisor $D_a\subset X_{a,b}$ (resp., on the toric divisor $D_b\subset X_{a,b}$). Denote by ${\mathcal M}_{0,m+m'+1}(X_{a,b},k,\beta,\beta',\bz,\bz')$ the space of isomorphism classes of stable maps of marked rational curves $\qquad$ $[\bn:(\PP^1,(\bu,\bu',u_0))\to X_{a,b}]$ such that $\bn_*[\PP^1]=k\Delta$ and
\begin{itemize}\item the $m$ marked points $\bu=\{u_i\}_{i=1,...,n}$ are taken by $\bn$ to the sequence $\bz$, while the $m'$ marked points $\bu'=\{u'_i\}_{i=1,...,n'}$ are taken by $\bn$ to the sequence $\bz'$, and $z_0:=\bn(u_0)\in D_{out}\subset X_{a,b}$,
\item $\bn^*(z_i)=\beta_iu_i$ for all $i=1,...,m$, $\bn^*(z'_i)=\beta'_iu'_i$ for all $i=1,...,m'$, and $\bn^*(z_0)=ku_0$.
\end{itemize}

\begin{lemma}\label{lrtv1}
If $\bz$ and $\bz'$ are in general position on $D_a$ and $D_b$, respectively, then ${\mathcal M}_{0,m+m'+1}(X_{a,b},\beta,\beta',\bz,\bz')$ is finite, and for each of its elements, the map $\bn:\PP^1\to X_{a,b}$ is a birational immersion onto a curve which is smooth along the toric divisors.
\end{lemma}

This is a particular case of \cite[Lemma 2.4]{IS}.

Now, let us be given the data $a,b,k,\alpha,\alpha'$ as in Section \ref{secrtv3}. In the real part $X_{a,b}(\R)$ of the toric surface $X_{a,b}(\C)$, take the nonnegative quadrant $X^+_{a,b}(\R)$ and denote by $D^+_a,D^+_b,D^+_{out}$ the segments of the toric divisors $D_1,D_b,D_{out}$ that form the boundary of $X^+_{a,b}(\R)$. Then choose a sequence $\bz$ of $\#\alpha$ generic points on $D^+_a$ and a sequence $\bz'$ of $\#\alpha'$ generic points on $D^+_b$.
Denote by $2\alpha$ and $2\alpha'$ the partitions of $2ka$ and $2kb$, respectively, obtained by the doubling of the summands of $\alpha$ and $\alpha'$. Consider the space ${\mathcal M}_{0,n}^\R(X_{a,b},2k,2\alpha,2\alpha',\bz,\bz')$
of equivariant isomorphism classes of stable conjugation-invariant maps of real $n$-marked rational curves $[\bn:(\PP^1,(\bu,\bu',u_0))\to X_{a,b}]$, where $n=\#\alpha+\#\alpha'+1$, $\bn_*[\PP^1]=2k\Delta$, 
and
\begin{itemize}\item $\#\bu=\#\alpha$, $\#\bu'=\#\alpha'$, $\bu\cup\bu'\subset\PP^1_\R$, $u_0\in\PP^1_\R$;
\item $\bn(\bu)=\bz$, $\bn(\bu')=\bz'$ so that $\bn^*(z_{ij})=2\alpha_{ij}u_{ij}$ and $\bn^*(z'_{ij})=2\alpha'_{ij}u_{ij}$ for all relevant $i,j$; furthermore $\bn(u_0)=z_0\in D^+_{out}$ so that  $\bn^*(z_0)=2ku_0$.
\end{itemize}
The choice of a component of $\PP^1_\C\setminus\PP^1_\R$ defines an orientation (called {\it complex orientation}) on $\PP^1_\R$ and on the immersed (by Lemma \ref{lrtv1}) circle $\bn(\PP^1_\R)\subset X^+_{a,b}(\R)$ of an element $\xi=[\bn(\PP^1,\bu,\bu',u_0)\to X_{a,b}]\in {\mathcal M}_{0,n}^\R(X_{a,b},2k,2\alpha,2\alpha',\bz,\bz')$. Denote by ${\mathcal M}_{0,n}^{\R,or}(X_{a,b},2k,2\alpha,2\alpha',\bz,\bz')$ the set of elements of ${\mathcal M}_{0,n}^\R(X_{a,b},2k,2\alpha,2\alpha',\bz,\bz')$ equipped with a complex orientation. By \cite[Theorem 3.1 and Definition 3.2]{Mi1} each element $\xi\in{\mathcal M}_{0,n}^{\R,or}(X_{a,b},2k,2\alpha,2\alpha',\bz,\bz')$ possesses a quantum index $QI(\xi)$ which in our situation belongs to $2\Z$ and ranges in $[-2k^2ab,2k^2ab]$.

To enumerate the elements of the set ${\mathcal M}_{0,n}^{\R,or}(X_{a,b},2k,2\alpha,2\alpha',\bz,\bz')$ which is finite by Lemma \ref{lrtv1}, we introduce a suitably modified Welschinger sign (cf. \cite[Formula (4)]{IS}). Choose an orientation of $\partial X^+_{a,b}(\R)$ induced by the form $dx\wedge dy$. For an element
$$\xi=[\bn:(\PP^1,\bu,\bu',u_0)\to X_{a,b}]\in{\mathcal M}_{0,n}^{\R,or}(X_{a,b},2k,2\alpha,2\alpha',\bz,\bz'),$$
put
\begin{equation}w^+(\xi)=(-1)^{e_+(\xi)}\cdot(-1)^r,\label{e-wel}\end{equation}
where \begin{itemize} \item $e_+(\bn(\PP^1))$ is the sum of the numbers $e(\xi,z)$ over all singular points $z$ of the curve $\bn(\PP^1)$ in the positive quadrant $X^+_{a,b}(\R)$, where $e(\xi,z)$ is the sum of the intersection numbers of the local complex conjugate branches of $\bn(\PP^1)$ centered at $z$ \footnote{If $\bn(\PP^1)$ is a nodal curve, then $e_+(\xi)$ is just the number of real elliptic nodes of $\bn(\PP^1)$ in $X^+_{a,b}(\R)$ ({\it elliptic} means equivalent to $x^2+y^2=0$ over the reals).}, \item $r$ is the number of those points in $\bz\cup\bz'\cup\{z_0\}$, where the intersection multiplicity of $\bn(\PP^1)$ with the toric divisor is divisible by $4$, and the complex orientation of $\bn(\PP^1_\R)$ at that point is opposite to the orientation of $\partial X^+_{a,b}(\R)$. \end{itemize}

By \cite[Theorem 2.5]{IS}, the number
$$W_0^\varkappa(a,b,k,2\alpha,2\alpha'):=\sum_{\renewcommand{\arraystretch}{0.6}
\begin{array}{c}
\scriptstyle{\xi\in{\mathcal M}_{0,n}^{\R,or}(X_{a,b},2k,2\alpha,2\alpha',\bz,\bz')}\\
\scriptstyle{QI(\xi)=\varkappa}\end{array}}w^+(\xi)$$
does not depend on the choice of generic sequences $\bz\subset D^+_a$, $\bz'\subset D^+_b$, for any even $\varkappa\in[-2k^2ab,2k^2ab]$.


\begin{proposition}\label{prtv2}
In the notations introduced in the preceding paragraphs, 
$$\sum_{\renewcommand{\arraystretch}{0.6}
\begin{array}{c}
\scriptstyle{-2k^2ab\le\varkappa\le2k^2ab}\\
\scriptstyle{\varkappa\in2\Z}\end{array}}(-1)^{k^2ab/2-\varkappa/4}\cdot W_0^\varkappa(a,b,k,2\alpha,2\alpha')\cdot q^{\varkappa/4}$$
\begin{equation}=\widehat N^{trop}_{\alpha,\alpha'}(\bP_{ka},\bP_{kb})(q)\cdot\prod_{i,j}[\alpha_{ij}]_q\cdot\prod_{i,j}[\alpha'_{ij}]_q
\cdot(q^{1/2}-q^{-1/2})^{\#\alpha+\#\alpha'-1}.\label{edop1}\end{equation}
\end{proposition}

\begin{proof}
This is, in fact, a modified version of \cite[Theorem 5.9]{Mi1}, see details in \cite[Remark 4.29]{IS}. Note also that the exponent $k^2ab/2-\varkappa/4$ of $(-1)$ in the above formula is an integer (see \cite[Proposition 4.5]{IS}).
\end{proof}

\section{The case of $q=-1$}\label{secrtv6}

In the family of algebras (\ref{ertv8}), consider the specialization at $q=-1$:
$$\C[\widehat x,\widehat x^{-1},\widehat y,\widehat y^{-1}][[t]],\quad\text{where}\ \widehat x\widehat y=-\widehat y\widehat x.$$
For heuristic reasons (see \cite{BG,GSh}), $q=-1$ should correspond to enumeration of real curves. We show that this indeed is the case.

Consider the partitions $\alpha$ and $\alpha'$ of $ka$ and $kb$, respectively, introduced in Section \ref{secrtv4}.
Recall that the tropical toric surface $X_{a,b}(\T)$ associated with the triangle $T_{a,b}=\conv\{(0,0),(a,0),(0,b)\}$ can be identified with $T_{a,b}$, while the tropical toric divisors are
$$D_a^{trop}=[(0,0),(a,0)],\ D_b^{trop}=[(0,0),(0,b)],\ D_{out}^{trop}=[(a,0),(0,b)]$$
(see \cite[Chapter 3]{MiR}). Sequences $\bx\subset D_a^{trop}$, resp. $\bx'\subset D_b^{trop}$, of $\#\alpha$, resp. $\#\alpha'$, distinct points in general position determine tropical constraints $\{L(\alpha_{ij})\}_{i,j}$ and $\{L(\alpha'_{ij})\}_{i,j})$, respectively, (see Section \ref{secrtv4}). These constraints determine a finite set ${\mathcal T}_0(\alpha,\alpha',\{L(\alpha_{ij})\}_{i,j},\{L(\alpha'_{ij})\}_{i,j})$ of trivalent plane rational tropical curves and the well-defined refined invariant $\widehat N^{trop}_{\alpha,\alpha'}(\bP_{ka},\bP_{kb})(q)$.

A relevant set of real algebraic curves to count is as follows. For each odd summand $\alpha_{ij}$ in $\alpha$, we pick a real point $z_{ij}\in D_a\subset X_{a,b}$ and for each even summand $\alpha_{ij}$ in $\alpha$, we pick a pair of complex conjugate points $(z_{ij},\overline z_{ij})\subset D_a$. Denote the obtained sequence by $\bz$. In a similar way, we construct a conjugation-invariant sequence $\bz'\subset D_b$. Denote by ${\mathcal MC}^\R_{0,n}(X_{a,b},k,\alpha,\alpha',\bz,\bz')$ \footnote{Note the difference in the definition of ${\mathcal M}_{0,n}^\R(X_{a,b},2k,2\alpha,2\alpha',\bz,\bz')$ from Section \ref{secrtv7} and of ${\mathcal MC}_{0,n}^\R(X_{a,b},k,\alpha,\alpha',\bz,\bz')$: in the former case, the constraint $\bz\cup\bz'$ is totally real, while in the latter case, it consists of real and pairs of complex conjugate points.}
the space of (equivariant) isomorphism classes of stable conjugation-invariant maps of real $n$-marked rational curves $\qquad$ $[\bn:(\PP^1,(\bu,\bu',u_0))\to X_{a,b}]$, where $\bn_*[\PP^1]=k\Delta$,
\begin{align*}n=&\#\bz+\#\bz'+1=\#\{\alpha_{ij}\equiv1\mod2\}+\#\{\alpha'_{ij}\equiv1\mod2\}\\ &+2\#\{\alpha_{i'j'}\equiv0\mod2\}
+2\#\{\alpha'_{i'j'}\equiv0\mod2\}+1,\end{align*}
and
\begin{enumerate}\item[(MC1)] the sequence $\bu$ (resp., $\bu'$) contains $\#\{\alpha_{ij}\equiv1\mod2\}$ (resp., $\#\{\alpha'_{ij}\equiv1\mod2\}$) real points and $\#\{\alpha_{i'j'}\equiv0\mod2\}$ (resp., $\#\{\alpha'_{i'j'}\equiv0\mod2\}$) pairs of complex conjugate points; the point $u_0$ is real;
\item[(MC2)] $\bn(\bu)=\bz$ so that $\bn^*(z_{ij})=\alpha_{ij}u_{ij}$ for all $i,j$ such that $\alpha_{ij}$ is odd, and $\bn^*(z_{ij})=\frac{\alpha_{ij}}{2}u_{ij}$ for all $i,j$ such that $\alpha_{ij}$ is even; in a similar way, we require $\bn(\bu')=\bz'$ and $\bn^*(z'_{ij})=\alpha'_{ij}u_{ij}$ or $\bn^*(z'_{ij})=\frac{\alpha'_{ij}}{2}u_{ij}$;
    \item[(MC3)] $z_0:=\bn(u_0)\in D_{out}$ so that $\bn^*(z_0)=ku_0$.
\end{enumerate}

Assuming that the sequences $\bz\subset D_a$, $\bz'\subset D_b$ are in general position, we derive from Lemma \ref{lrtv1}
that ${\mathcal MC}^\R_{0,n}(X_{a,b},k,\alpha,\alpha',\bz,\bz')$ is finite. For each element $\xi=[\bn:(\PP^1,(\bu,\bu',u_0))\to X_{a,b}]\in{\mathcal MC}^\R_{0,n}(X_{a,b},k,\alpha,\alpha',\bz,\bz')$,  
we define the {\it Welschinger sign} (cf. (\ref{e-wel}))
$$w(\xi)=(-1)^{e(\xi)},$$
where $e(\xi)$ is the sum of the numbers $e(\xi,z)$ over all real singular points of $\bn(\PP^1)$.

\begin{proposition}\label{prtv1}
Suppose that $a,b\in\N$ are coprime, $k\in\N$ is odd, and $\alpha,\alpha'$ are ordered partitions of $ka$ and $kb$, respectively. Let conjugation-invariant sequences $\bz\subset D_a$ and $\bz'\subset D_b$ be associated with the partitions $\alpha$ and $\alpha'$, respectively, as in the preceding paragraphs and are in general position on $D_a$ and $D_b$.
Then
\begin{equation}\widehat N^{trop}_{\alpha,\alpha'}(\bP_{ka},\bP_{kb})(-1)=(-1)^{(k-1)/2}\sum_{\xi\in{\mathcal MC}^\R_{0,n}(X_{a,b},k,\alpha,\alpha',\bz,\bz')}w(\xi).\label{ertv10}\end{equation}
\end{proposition}

\begin{proof} {\bf(1)} In Proposition \ref{lrtv3} (Section \ref{app}), we show that the right-hand side of (\ref{ertv10}) does not depend on the choice of generic sequences $\bz,\bz'$. Thus, here we prove relation (\ref{ertv10}) under the assumption that $\bz,\bz'$ are in ``a neighborhood of a general tropical limit". This means that $\bz,\bz'$ are taken in a one-parameter family
$$z_{ij}=c_{ij}t^{x_{ij}}\in D_a,\ x_{ij}\in\Q^*,\ c_{ij}\in\begin{cases}\R^*,\quad & \alpha_{ij}\equiv1\mod2,\\
\C\setminus\R,\quad &\alpha_{ij}\equiv0\mod2,\end{cases}$$
$$z'_{ij}=c'_{ij}t^{y_{ij}}\in D_b,\ y_{ij}\in\Q^*,\ c'_{ij}\in\begin{cases}\R^*,\quad & \alpha'_{ij}\equiv1\mod2,\\
\C\setminus\R,\quad &\alpha'_{ij}\equiv0\mod2,\end{cases}$$
for all relevant $i,j$, where $0<t\ll1$, and all constants $x_{ij}$, $y_{ij}$, $c_{ij}$, $c'_{ij}$ are in general position subject to the restriction in the above formulas. All sequences $\bz,\bz'$ in the family are in general position on $D_a$ and $D_b$, and, by \cite[Theorem 6]{Mi} (see also \cite[Proposition 6.1]{Sh05} and \cite[Theorem 3.1]{Sh06}) each element $\xi\in{\mathcal MC}^\R_{0,n}(X_{a,b},k,\alpha,\alpha',\bz,\bz')$ yields a family that defines a tropical limit as $t\to0$, and this limit is encoded by a trivalent rational tropical curve $\Gamma\in{\mathcal T}_0(\alpha,\alpha',\{L(\alpha_{ij})\}_{i,j},\{L(\alpha'_{ij})\}_{i,j})$. 

\smallskip{\bf(2)} Suppose that all summands in $\alpha$ and $\alpha'$ are odd and all points in $\bz$ and $\bz'$ are real. Observe that the value $\widehat N^{trop}_{\alpha,\alpha'}(\bP_{ka},\bP_{kb})(-1)$ is well defined: the factors in (\ref{ertv11}) associated with the vertices are
$$\frac{\bi^{\mu(\Gamma,V)}-\bi^{-\mu(\Gamma,V)}}{\bi-\bi^{-1}}=\begin{cases}0,\quad &\text{if}\ \mu(\Gamma,V)\ \text{is even},\\
\bi^{\mu(\Gamma,V)-1},\quad &\text{if}\ \mu(\Gamma,V)\ \text{is odd},\end{cases}\quad\text{where}\ \bi^2=-1,$$ while the factors associated with the constrained ends contribute, in view of the assumption that all $\alpha_{ij},\alpha'_{ij}$ odd,
$$\prod_{i,j}\bi^{1-\alpha_{ij}}\cdot\prod_{i,j}\bi^{1-\alpha'_{ij}}.$$ In total, the contribution of a curve $\Gamma\in{\mathcal T}_0(\alpha,\alpha',\{L(\alpha_{ij})\}_{i,j},\{L(\alpha'_{ij})\}_{i,j})$ to $\qquad$ $\widehat N^{trop}_{\alpha,\alpha'}(\bP_{ka},\bP_{kb})(-1)$ equals $0$ if there is at least one even $\mu(\Gamma,V)$, or equals
\begin{equation}\bi^r,\quad r=\sum_{V\in\Gamma^0}(\mu(\Gamma,V)-1)+\#\alpha+\#\alpha'-ka-kb,\label{ertv12}\end{equation}
otherwise.

By \cite[Theorem 6]{Mi} (see also \cite[Proposition 6.1]{Sh05}), 
the sum of Welschinger signs of the elements $\xi\in{\mathcal MC}^\R_{0,n}(X_{a,b},k,\alpha,\alpha',\bz,\bz')$ tropicalizing to $\Gamma$
equals $0$ if there is at least one even $\mu(\Gamma,V)$, or equals
$$\prod_{V\in\Gamma^0}\bi^{2I(\Gamma,V)}$$
otherwise, where $\Gamma^0$ is the set of the vertices of $\Gamma$, $I(\Gamma,V)$ is the number of integral points in the lattice triangle dual to $V$. By Pick's theorem,
$$2I(\Gamma,V)=\mu(\Gamma,V)-P(\Gamma,V)+2,$$ where $P(\Gamma,V)$ is the lattice perimeter of the dual triangle (equivalently, the total weight of the edges of $\Gamma$ incident to $V$). Since $\Gamma$ is trivalent,
$$3\#\Gamma^0=2\#\Gamma^1_{fin}+\#\alpha+\#\alpha'+1,$$
where $\Gamma^1_{fin}$ is the set of finite edges of $\Gamma$, and since all finite edges of $\Gamma$ have odd weights, we obtain (cf. \cite[Proposition 2.3(5)]{IM}) that
\begin{equation}\sum_{V\in\Gamma^0}2I(\Gamma,V)\equiv\sum_{V\in\Gamma^0}\mu(\Gamma,V)-2\#\Gamma^1_{fin}-ka-kb-k+2\#\Gamma^0
\label{ertv15}\end{equation}
$$=\sum_{V\in\Gamma^0}(\mu(\Gamma,V)-1)+\#\alpha+\#\alpha'+1-ka-kb-k\mod4.$$
Comparing with (\ref{ertv12}), we derive (\ref{ertv10}).

\smallskip
{\bf(3)} Now we allow arbitrary summands in $\alpha$ and $\alpha'$.

Observe that
\begin{equation}(\bi+\tau)^s-(\bi+\tau)^{-s}=2s\bi^{s-1}\tau+O(\tau^2)\quad\text{as}\ \tau\to0,\ s\equiv0\mod2.
\label{ertv14}\end{equation}

Let $G(\Gamma)$ be a subgraph of $\Gamma\in{\mathcal T}_0(\alpha,\alpha',\{L(\alpha_{ij})\}_{i,j},\{L(\alpha'_{ij})\}_{i,j})$ formed by the edges of even weight. In particular, the infinite edges of $G(\Gamma)$ are the ends of $\Gamma$ merging to the points $x_{ij}\in\bx$, resp., $x'_{ij}\in\bx'$, corresponding to even $\alpha_{ij}$, resp. even $\alpha'_{ij}$. Note also that $G(\Gamma)$ does not contain bivalent vertices, and each vertex $V\in G(\Gamma)\cap\Gamma^0$ satisfies $\mu(\Gamma,V)\equiv0\mod2$.

Each component $G_0$ of $G(\Gamma)$ is a tree with $v_1(G_0)\ge1$ univalent vertices and $e(G_0)\ge1$ ends, and the total number of the vertices in $G_0$ equals $e(G_0)+v_1(G_0)-1$. It follows from (\ref{ertv14}) that if $v_1(G_0)>1$, then
$$\widehat N^{trop}_{\alpha,\alpha'}(\bP_{ka},\bP_{kb})(-1)=\lim_{\tau\to0}
\widehat N^{trop}_{\alpha,\alpha'}(\bP_{ka},\bP_{kb})(-1+\tau)=0.$$
Suppose now that $v_1(G_0)=1$ for each component $G_0$ of $G(\Gamma)$. Then
$$\widehat N^{trop}_{\alpha,\alpha'}(\bP_{ka},\bP_{kb})(-1)=\lim_{\tau\to0}
\widehat N^{trop}_{\alpha,\alpha'}(\bP_{ka},\bP_{kb})(-1+\tau)$$
$$=\bi^r\cdot\prod_{V\in G(\Gamma)\cap\Gamma^0}\mu(\Gamma,V)\cdot\prod_{\alpha_{ij}\equiv0\mod2}\frac{1}{\alpha_{ij}}\cdot\prod_{\alpha'_{ij}
\equiv0\mod2}\frac{1}{\alpha'_{ij}},$$
\begin{equation}r=\sum_{V\in\Gamma^0}(\mu(\Gamma,V)-1)+\#\alpha+\#\alpha'-ka-kb.\label{ertv16}\end{equation}

By \cite[Theorem 3.1]{Sh06} there is a correspondence between tropical curves $\Gamma\in{\mathcal T}_0(\alpha,\alpha',\{L(\alpha_{ij})\}_{i,j},\{L(\alpha'_{ij})\}_{i,j})$ and elements $\xi\in{\mathcal MC}^\R_{0,n}(X_{a,b},k,\alpha,\alpha',\bz,\bz')$ such that the sum of the Welschinger signs of those $\xi$ which tropicalize to a given $\Gamma$
\begin{itemize}\item either equals (cf. \cite[Lemmas 2.3, 2.4, 2.5, 2.7, and 2.8]{Sh06})
$$\bi^s\cdot\prod_{V\in G(\Gamma)\cap\Gamma^0}\mu(\Gamma,V)\cdot\prod_{\alpha_{ij}\equiv0\mod2}\frac{1}{\alpha_{ij}}\cdot\prod_{\alpha'_{ij}
\equiv0\mod2}\frac{1}{\alpha'_{ij}},$$
\begin{equation}s=\sum_{V\in\Gamma^0}2I(\Gamma,V)+2v_3(G(\Gamma)),\label{ertv17}\end{equation}
where $v_3(G(\Gamma))$ is the number of trivalent vertices of $G(\Gamma)$, provided that $v_1(G_0)=1$ for each component $G_0$ of $G(\Gamma)$,
\item or equals $0$, otherwise.
\end{itemize}
It remains to consider the case of $\Gamma$ such that each component of $G(\Gamma)$ contains only one univalent vertex, and then to compare $r$ and $s$ in formulas (\ref{ertv16}) and (\ref{ertv17}). To this end, we observe that
$$\sum_{V\in\Gamma^0}P(\Gamma,V)\equiv 2(\text{number of finite edges of}\ \Gamma\ \text{with odd weight})$$
$$+ka+kb+k\mod4$$
and that the number of finite edges of $\Gamma$ with even weight equals $v_3(G(\Gamma))$. Hence (cf. the computation in part {\bf(2)} of the proof)
$$s\equiv\sum_{V\in\Gamma^0}(\mu(\Gamma,V)-1)+\#\alpha+\#\alpha'+1-ka-kb-k\mod4,$$
and the relation (\ref{ertv10}) follows.
\end{proof}

\begin{remark}\label{rrtv2}
(1) A reasonable enumerative interpretation of the invariants $\qquad$ $\widehat N^{trop}_{\alpha,\alpha'}(\bP_{ka},\bP_{kb})(-1)$ is restricted to the case of odd $k$, covered in Proposition \ref{prtv1}. Indeed, if $k$ is even, then
\begin{itemize}\item in case of all summands in $\alpha$ and $\alpha'$ even, one obtains (see (\ref{ertv14}))
$$\lim_{\tau\to0}\widehat N^{trop}_{\alpha,\alpha'}(\bP_{ka},\bP_{kb})(-1+\tau)=\infty,$$
since the number of vertices in $\Gamma$ is one less than $\#\alpha+\#\alpha'$,
\item if there are odd summands in $\alpha$ and/or $\alpha'$, the number of vertices $V\in\Gamma^0$ such that $\mu(\Gamma,V)$ is even appears to be greater than the total number of even summands in $\alpha$ and $\alpha'$, and hence
    $$\lim_{\tau\to0}\widehat N^{trop}_{\alpha,\alpha'}(\bP_{ka},\bP_{kb})(-1+\tau)=0.$$
\end{itemize}

(2) If all $\alpha_{ij}$'s and $\alpha'_{ij}$'s as well as $k$ are odd, then one can recover the invariant $\widehat N^{trop}_{\alpha,\alpha'}(\bP_{ka},\bP_{kb})(-1)$ from the invariants $W_0^\varkappa(a,b,k,2\alpha,2\alpha')$, $\varkappa\in2\Z$, discussed in Section \ref{secrtv7}, since under the above condition, the factors next to $\widehat N^{trop}_{\alpha,\alpha'}(\bP_{ka},\bP_{kb})(q)$ in (\ref{edop1}) do not vanish at $q=-1$.
\end{remark}

\section{Relative Welschinger invariants}\label{app}

The following statement completes the proof of Proposition \ref{prtv1}.

\begin{proposition}\label{lrtv3} In the notation and under the hypotheses of Proposition \ref{prtv1}, the number
$$WC_0(X_{a,b},k,\alpha,\alpha',\bz,\bz'):=\sum_{\xi\in{\mathcal MC}^\R_{0,n}(X_{a,b},k,\alpha,\alpha',\bz,\bz')}w(\xi)$$
does not depend on the choice of generic sequences $\bz\subset D_a$, $\bz'\subset D_b$ (subject to restrictions imposed in the definition of ${\mathcal MC}^\R_{0,n}(X_{a,b},k,\alpha,\alpha',\bz,\bz')$).
\end{proposition}

\begin{proof}
{\it Step 1.}
It is enough to verify the invariance of $WC_0(X_{a,b},k,\alpha,\alpha',\bz,\bz')$ in the following moves:
\begin{enumerate}\item[(M1)] One of the real points of $\bz$ jumps from the positive open interval of $D_a(\R)$ to the negative one, or vice versa. A similar move for real points of $\bz'$.
\item[(M2)] The sequences $\bz\subset D_a$ and $\bz'\subset D_b$ vary in a smooth family depending on one real parameter so that they remain disjoint from the intersection points of the toric divisors, the real points remain real, and the pairs of nonreal complex conjugate points remain pairs of nonreal complex conjugate points.
\end{enumerate}

Consider the move (M1). Given a generic sequence $\bz\subset D_a$, it yields a sequence of signs $\vec\eps=\{\eps_{ij}\}_{i,j}$ associated with odd $\alpha_{ij}$'s 
and defined so that $\eps_{ij}=1$ or $-1$ according as the point $z_{ij}$ belongs to the positive or the negative interval in $D_a(\R)$. Similarly, $\bz'$ yields a sequence of signs $\vec\eps'$.
We claim that for any pair of sign sequences $\vec\eps$, $\vec\eps'$, there exist a pair of generic sequences $\bz\subset D_a$, $\bz'\subset D_b$ that yield $\vec\eps$, $\vec\eps'$, respectively, and such that $WC_0(X_{a,b},k,\alpha,\alpha',\bz,\bz')$ takes the same value for all these pairs of sequences $\bz$, $\bz'$. Indeed, we can take $\bz$, $\bz'$ in a neighborhood of the tropical limit, and then each value $WC_0(X_{a,b},k,\alpha,\alpha',\bz,\bz')$ will be equal to $(-1)^{(k-1)/2}N^{trop}_{\alpha,\alpha'}(\bP_{ka},\bP_{kb})(-1)$ as shown in parts (2) and (3) of the proof of Proposition \ref{prtv1}.

\smallskip{\it Step 2.} For given sequences of signs $\vec\eps$ and $\vec\eps'$, the space ${\mathcal Z}(\vec\eps,\vec\eps')$ of pairs of the corresponding sequences $(\bz,\bz')$ is connected and has real dimension
$$n-1=\#\{\alpha_{ij}\ \text{odd}\}+2\#\{\alpha_{ij}\ \text{even}\}+\#\{\alpha'_{ij}\ \text{odd}\}+2\#\{\alpha'_{ij}\ \text{even}\}.$$
Denote by $\overline{\mathcal MC}^\R_{0,n}(X_{a,b},k,\alpha,\alpha',\vec\eps,\vec\eps')$ 
the closure in $\overline{\mathcal M}_{0,n}(X_{a,b},k\Delta)$ of the union of the sets ${\mathcal MC}^\R_{0,n}(X_{a,b},k,\alpha,\alpha',\bz,\bz')$, $(\bz,\bz')\in{\mathcal Z}(\vec\eps,\vec\eps')$, satisfying the conclusions of Lemma \ref{lrtv1}.
This space contains the union of top-dimensional {\it regular chambers}, where the fibers $\overline{\mathcal MC}^\R_{0,n}(X_{a,b},k,\alpha,\alpha',\bz,\bz')$ of $\overline{\mathcal MC}^\R_{0,n}(X_{a,b},k,\alpha,\alpha',\vec\eps,\vec\eps')$ over the pairs $(\bz,\bz')$ are finite and consists of elements 
satisfying the conclusion of Lemma \ref{lrtv1}.
The regular chambers are separated by walls of codimension one. Below we classify these walls and show that in all wall-crossing events (as well as in variation inside a regular chamber), the number $WC_0(X_{a,b},k,\alpha,\alpha',\bz,\bz')$ does not change.

\smallskip
{\it Step 3.} Consider a path $\{(\bz(\tau),\bz'(\tau))\}_{0\le\tau\le1}$ joining two generic pairs of sequences $(\bz(0),\bz'(0))$ and $(\bz(1),\bz'(1))$. We can suppose that the points of the sequences $\bz(\tau)$ and $\bz'(\tau)$ never enter the $\eta$-neighborhood $U_\eta$ of the intersection points of the toric divisors and the complex conjugate points of $\bz(\tau),\bz'(\tau)$ never enter the $\eta$-neighborhood $V_\eta$ of $X_{a,b}(\R)$, for some fixed $\eta>0$.

We claim that the elements $$\xi=[\bn:(\widehat C,\bu,\bu',u_0)\to X_{a,b}]\in\overline{\mathcal MC}^\R_{0,n}(X_{a,b},k,\alpha,\alpha',\vec\eps,\vec\eps')$$ that belong to the fibers of the projection $\overline{\mathcal MC}^\R_{0,n}(X_{a,b},k,\alpha,\alpha',\vec\eps,\vec\eps')\to{\mathcal Z}(\vec\eps,\vec\eps')$ over generic members $(\bz,\bz')$ of $(n-2)$-dimensional strata in the complement to the regular chambers, either satisfy the conclusion of Lemma \ref{lrtv1}, or are as follows (cf. \cite[Lemma 2.15]{IS}):
\begin{enumerate}\item[(i)] either $\widehat C=\PP^1$, $\bn:\PP^1\to X_{a,b}$ is birational onto its image, which is smooth along the toric divisors, and $\bz=\bn(\bu)$ consists of $n-1$ distinct points;
\item[(ii)] or $\widehat C=\PP^1$, $\bn:\PP^1\to X_{a,b}$ is an immersion onto its image $\bn(\PP^1)=C$, which is smooth along the toric divisors except for one real point $z_{ij}=z_{i'j'}\in\bz$, where $C$ has two smooth local branches; the sequence $\bz\cup\bz'$ contains only $n-2$ distinct points;
    \item[(iii)] or $\widehat C=\widehat C_1\cup\widehat C_2$, $\widehat C_1\simeq\widehat C_2\simeq\PP^1$, $u_{ij},u_{i'j'}\in\widehat C_2$, $\bn(\widehat C_2)=z_{ij}=z_{i'j'}$, $\bn:\widehat C_1\to X_{a,b}$ is an immersion with the image $C$ that is smooth along the toric divisors and intersects $D_a$ at the real point $z_{ij}=z_{i'j'}$ with multiplicity $\alpha_{ij}+\alpha_{i'j'}$; the sequence $\bz\cup\bz'$ contains $n-2$ distinct points;
        \item[(iv)] or $\widehat C=\PP^1$, $\bz$ consists of $n-1$ distinct points, $\bn:\widehat C\to X_{a,b}$ is an immersion outside one real point $u_{ij}$, and $C=\bn(\widehat C)$ has a singularity of type $A_{2s}$ at $z_{ij}$, where $\alpha_{ij}=2s+1$;
            \item[(v)] or $\widehat C=\widehat C_1\cup\widehat C_2\cup\widehat C_3$, $\widehat C_1\simeq\widehat C_2\simeq 
                \PP^1$, $\widehat C_1\cap\widehat C_2=\emptyset$, $\widehat C_1\cap\widehat C_3$ and $\widehat C_2\cap\widehat C_3$ are one-point sets, $u_0\in\widehat C_3$,
                $\bn_*[\widehat C_1]=k_1\Delta$, $\bn_*[\widehat C_2]=k_2\Delta$, $0<k_1<k_2<k$, $k_1+k_2=k$, $\bn(\widehat C_3)=z_0\in D_{out}$; the curves $C_1=\bn(\widehat C_1)$ and $C_2=\bn(\widehat C_2)$ are immersed and they are smooth along the toric divisors; the sequence $\bz\cup\bz'$ contains $n-1$ distinct points.
\end{enumerate}
In the next three steps we verify this list.

\smallskip
{\it Step 4.}
Let us recall that the configuration $\bz\cup\bz'\cup\{z_0\}$ is subject to the so-called Menelaus condition, see \cite[Page 139]{Mi1} and \cite[Section 2.1.2]{IS}. It means that the product of the coordinates of the points of $\bz\cup\bz'\cup\{z_0\}$ raised to the powers equal to the intersection multiplicites with the toric divisors at the corresponding points is an absolute constant. This yields the following:
\begin{itemize}\item since $k$ is odd, the configuration $\bz\cup\bz'$ uniquely determines the point $z_0\in D_{out}$;
\item since the family in move (M2) can be chosen generic, at most two points of $\bz\cup\bz'$ can collide at a moment and these points are real; it, furthermore, yields that if an element $[\bn:(\widehat C,\bu,\bu',u_0)\to X_{a,b}]\in\overline{\mathcal MC}^\R_{0,n}(X_{a,b},k,\alpha,\alpha',\vec\eps,\vec\eps')$ occurs in the family induced by the path $\{(\bz(\tau),\bz'(\tau))\}_{0\le\tau\le1}$, then $\widehat C$ contains at most two noncontractible components, and if there are two of them, then the subsequences of $\bz\cup\bz'$ corresponding to these components are in general position on $D_a\cup D_b$ subject to one condition that they determine the same point $z_0\in D_{out}$.
\end{itemize}

\smallskip
{\it Step 5.}
Next, we note that $\widehat C$ cannot contain components mapped onto toric divisors since the intersections of the image curve with the toric divisors are far from intersection points of the toric divisors, and all three toric divisors cannot split off since otherwise the arithmetic genus of $\widehat C$ would be positive.

We also claim is that no component of $\widehat C$ multiply covers a curve $C'\subset X_{a,b}$. Indeed, assume that $\widehat C'\subset\widehat C$ is one of two noncontractible components, and $\bn:\widehat C'\to C'$ is an $s$-multiple covering, $s\ge2$. In view of the observations made in Step 4 and in view of the conditions (MC2), (MC3) in the definition of ${\mathcal MC}_{0,n}^\R(X_{a,b},k\alpha,\alpha',\bz,\bz')$ (see Section \ref{secrtv6}), the map $\bn:\widehat C'\to C'$ must be ramified at each intersection point of $C'$ with the toric divisors with the same ramification index, which contradicts the Riemann-Hurwitz formula. Assume now that $\widehat C'\subset\widehat C$ is the only noncontractible component, and $\widehat C'\to C'$ is an $s$-multiple covering, $s\ge2$. Then this map must be ramified with the ramification index $s$ at each intersection point of $C'$ with the toric divisors, except for the point of collision $z=z_{i_1j_1}=z_{i_2j_2}$ (or $z=z'_{i_1j_1}=z'_{i_2j_2}$) for some $(i_1,j_1)\ne(i_2,j_2)$. It follows that $s$ divides $k$, and hence $s\ge3$. However, in the deformation induced by variation along the path $\{(\bz(\tau),\bz'(\tau))\}_{0\le\tau\le1}$ some $s-1\ge2$ local branches of $\bn:\widehat C'\to C'$ centered at $z$ must glue up yielding a component of a positive genus in contradiction to the rationality of the considered curves.

\smallskip{\it Step 6.} Suppose that $\widehat C$ contains two noncontractible components. Then by Lemma \ref{lrtv1} and the observations in Step 4, we are left with the case (v) of the list in Step 3.

Suppose that $\widehat C$ contains just one noncontractible component $\widehat C'$ and that all the points of $\bz\cup\bz'$ are distinct. We claim that this is the situation either of the case (i), or of the case (iv) in the list of Step 3. Indeed, if it is not the case (i), then $\bn:\widehat C'\to X_{a,b}$ is unibranch at each point of $\bz\cup\bz'\cup\{z_0\}$ and singular in at least one of these points. Note that then $n\ge4$, since otherwise, $\bn(\widehat C')$ is smooth along the toric divisors (see \cite[Lemma 3.5]{Sh05}).
By our assumption, $\bz\cup\bz'$ is a generic member of an $(n-2)$-dimensional family of such configurations, and we can suppose that this is a smooth element of the family and that the variation of $\bz\cup\bz'$ along the family induced an equisingular variation of $\bn(\widehat C')$. Suppose that $z_{i_0j_0}$ is a real singular point of $\bn(\widehat C')$, the center of a local singular branch $Q_0$. Fixing the position of this point, we get an equisingular family of dimension $\ge n-3\ge1$. Applying the inequalities of \cite[Theorem 2]{GuS} in the form of \cite[Lemma 2.1]{IKS18}, we obtain
$$ka+kb+k\ge2+\left(\sum_{(i,j)\ne(i_0,j_0)}(\alpha_{ij}-1)+\sum_{i,j}(\alpha'_{ij}-1)+(k-1)\right)+\alpha_{i_0j_0}$$
$$+(\ord Q_0-1)+\sum_Q(\ord Q-1)+n-4$$
$$=ka+kb+k+\left((\ord Q_0-1)+\sum_Q(\ord Q-1)-1\right),$$
where $Q$ ranges over all singular local branches of $\bn(\widehat C')$ in the torus $(\C^*)^2\subset X_{a,b}$. It follows that $\bn(\widehat C')$ is immersed outside toric divisors, and $\ord Q_0=2$. In particular, the singularity at $z_{i_0j_0}$ is of type $A_{2s}$, and $\alpha_{i_0j_0}=2s+1$ since $z_{i_0j_0}$ is real and $\alpha_{i_0j_0}$ must be odd. The same conclusion holds when the point $z_0$ is singular. Suppose that there are at least two singular points among $\bz\cup\bz'\cup\{z_0\}$, say $z_{i_1j_1}$ and $z_{i_2j_2}$ (real or complex conjugate), the centers of singular local branches $Q_1,Q_2$, respectively. Then $n\ge5$, since otherwise, the curve $\bn(\widehat C')$ would have a parametrization (in the suitable affine coordinates)
\begin{equation}x_1=c_1(t-t_1)^{ka},\quad x_2=c_2(t-t_2)^{p}(t-t_3)^q,\quad t_1\ne t_2\ne t_3\in\C^*,\ p+q=kb,\label{edop3}\end{equation} and one would easily verify that a singular local branch on toric divisors may occur only at $t=t_1$. So, we fix the position of $z_{i_1j_1},z_{i_2j_2}$ obtaining an equisingulat family of dimension $\ge n-4\ge1$, and again apply the inequalities of \cite[Theorem 2]{GuS} in the form of \cite[Lemma 2.1]{IKS18}:
$$ka+kb+k\ge2+\left(\sum_{(i,j)\ne(i_1,j_1),(i_2j_2)}(\alpha_{ij}-1)+\sum_{i,j}(\alpha'_{ij}-1)+(k-1)\right)$$
$$+\alpha_{i_1j_1}+\alpha_{i_2j_2}+(\ord Q_1-1)+(\ord Q_2-1)+n-5$$
\begin{equation}=ka+kb+k+\left((\ord Q_1-1)+(\ord Q_2-1)-1\right),\label{edop2}\end{equation}
which is a contradiction. Thus, the conditions of the case (iv) are verified.

Suppose that $\widehat C$ contains exactly one noncontractible component $\widehat C'$ and that two the points of $\bz\cup\bz'$ collide, say $z_{i_1j_1}=z_{i_2j_2}+:z$. Note that the configuration $(\bz\setminus\{z_{i_1j_1},z_{i_2j_2}\})\cup\{z\}\cup\bz'$ varies in an $(n-2)$-dimensional family; hence, the latter configuration must be in general position on $D_a\cup D_b$. If in addition, $\bn:\widehat C'\to X_{a,b}$ is unibranch at each point on the toric divisors, Lemma \ref{lrtv1} will imply the conditions of the case (iii) in the list of Step 3. The remaining option is that $\bn:\widehat C'\to X_{a,b}$ has two local branches $Q_1,Q_2$ centered at $z$ and is unibranch at each other point of $\bz\cup\bz'\cup\{z_0\}$. Fixing the position of $z$, we obtain an $(n-3)$-dimensional equisingular family of curves in $X_{a,b}$ and again apply the inequalities of \cite[Theorem 2]{GuS} in the form of \cite[Lemma 2.1]{IKS18}:
$$ka+kb+k\ge2+\left(\sum_{(i,j)\ne(i_1,j_1),(i_2j_2)}(\alpha_{ij}-1)+\sum_{i,j}(\alpha'_{ij}-1)+(k-1)\right)$$
$$+\alpha_{i_1j_1}+\alpha_{i_2j_2}+(\ord Q_1-1)+(\ord Q_2-1)+\sum(\ord Q-1)+n-4$$
$$=ka+kb+k+(\ord Q_1-1)+(\ord Q_2-1)+\sum(\ord Q-1),$$ and hence $Q_1,Q_2$ are smooth, and the curve $\bn(\widehat C')$ is immersed outside $\qquad$ $\qquad$ \mbox{$(\bz\setminus\{z_{i_1j_1},z_{i_2j_2}\})\cup\bz'\cup\{z_0\}$}. Assuming that at least two points in $(\bz\setminus\{z_{i_1j_1},z_{i_2j_2}\})\cup\bz'\cup\{z_0\}$ are singular, we proceed as in the preceding paragraph and come to the contradiction via the chain of relations (\ref{edop2}). Assuming that just one point in $(\bz\setminus\{z_{i_1j_1},z_{i_2j_2}\})\cup\bz'\cup\{z_0\}$ is singular, we first note that $n\ge5$. Indeed, otherwise we would get the parametrization (\ref{edop3}) for the curve $\bn(\widehat C')$, in which $t_1,t_2,t_3\in\R^*$ are distinct and satisfy the relations
$$(t_2-t_1)^{ka}=(t_3-t_1)^{ka},\quad p(t_2-t_1)^{p-1}+q(t_3-t_1)^{q-1}=0.$$
This would imply that $(t_2-t_1)=-(t_3-t_1)$, $p=q$, and $ka$ is even. Further on, $kb=2p$ would be even too contrary to the initial assumptions that $k$ is odd and $\gcd(a,b)=1$. Taking into account that $n\ge5$, we fix the position of the (real) singular point $z_{i_0j_0}\in(\bz\setminus\{z_{i_1j_1},z_{i_2j_2}\})\cup\bz'\cup\{z_0\}$ and of $z$ and then apply the inequalities of \cite[Theorem 2]{GuS} in the form of \cite[Lemma 2.1]{IKS18}:
$$ka+kb+k\ge2+\left(\sum_{(i,j)\ne(i_0,j_0),(i_1,j_1),(i_2j_2)}(\alpha_{ij}-1)+\sum_{i,j}(\alpha'_{ij}-1)+(k-1)\right)$$
$$+\alpha_{i_0j_0}+\alpha_{i_1j_1}+\alpha_{i_2j_2}+(\ord Q_0-1)+n-5$$
$$=ka+kb+k+(\ord Q_0-1),$$
which is a contradiction. Thus, the conditions of the case (ii) in the list of Step 3 are verified.

\smallskip
{\it Step 7.} Now we prove the constancy of the function $WC_0(X_{a,b},k,\alpha,\alpha',\bz,\bz')$ along the path $\{(\bz(\tau),\bz'(\tau))\}_{0\le\tau\le1}$ both, between wall-crossing events and in each wall-crossing event.

{\it7.1.} Considerations of moves between wall-crossing events and of the wall-crossing of type (i) are the same, and the constancy of $WC_0(X_{a,b},k,\alpha,\alpha',\bz,\bz')$ follows from \cite[Lemma 15]{IKS}, provided, we establish the following transversality condition (cf. \cite[Lemma 13]{IKS}): The linear system in $|{\mathcal O}_{X_{a,b}}(k\Delta)|$ formed by the curves intersecting the toric divisors at the points of $\bz$, $\bz'$ with multiplicities determined by $\alpha$ and $\alpha'$ (i.e., $\alpha_{ij}$ or $\alpha'_{ij}$ at a real point, and $\alpha_{ij}/2$ or $\alpha'_{ij}/2$ at a nonreal point), and at the point $z_0$ with multiplicity $k-1$, respectively, intersects transversally at $C=\bn(\widehat C)$ with the germ at $C$ of the family of rational curves in $|{\mathcal O}_{X_{a,b}}(k\Delta)|$. It amounts to the relation (cf. \cite[Proof of Lemma 13]{IKS})
\begin{equation}H^1(\PP^1,{\mathcal O}_{\PP^1}(-D-\sum_{i,j}\alpha_{ij}u_{ij}-\sum_{i,j}\alpha'_{ij}u'_{ij}-(k-1)u_0)\otimes\bn^*{\mathcal O}_{X_{a,b}}(k\Delta))=0,\label{ertv18}\end{equation} where $D$ is the double point divisor of degree
$$\deg D=2\sum_{z\in\Sing(C)}\delta(C,z)=C^2+CK_{X_{a,b}}+2=k^2ab-ka-kb-k+2.$$
Finally,
$$\deg({\mathcal O}_{\PP^1}(-D-\sum_{i,j}\alpha_{ij}u_{ij}-\sum_{i,j}\alpha'_{ij}u'_{ij}-(k-1)u_0)\otimes\bn^*{\mathcal O}_{X_{a,b}}(k\Delta))$$
$$=-k^2ab+ka+kb+k-2-ka-kb-(k-1)+k^2ab=-1>(2g-2)\big|_{g=0}=-2,$$
and hence (\ref{ertv18}) follows.

{\it7.2.} Suppose that $\xi\in\overline{\mathcal MC}^\R_{0,n}(X_{a,b},k,\alpha,\alpha',\vec\eps,\vec\eps')$ is a generic element of the family described in item (ii).
We can assume that the considered path is locally like that: all the points of $\bz'$ and all the points of $\bz$ except for $z_{ij}$ are fixed in general position, and the point $z_{ij}$ is mobile. Following the lines of \cite[Proof of Lemma 2.16, part (3)]{IS}, we consider two smooth families of curves in $|{\mathcal O}_{X_{a,b}}(k\Delta)|$:
\begin{itemize}\item the germ at $C=\bn(\widehat C)$ of the family of curves $C'\in|{\mathcal O}_{X_{a,b}}(k\Delta)|$ whose intersection with a small neighborhood of $z_{i'j'}$ is the union of two smooth branches, one branch intersects $D_a$ at $z_{i'j'}$ with multiplicity $\alpha_{i'j'}$ and  another branch intersects $D_a$ at a point close to $z_{i'j'}$ with multiplicity $\alpha_{ij}$;
\item the germ at $C$ of the family of rational curves intersecting the toric divisors at $\bz\setminus\{z_{ij},z_{i'j'}\}$ and at $\bz'$ with multiplicities determined by $\alpha$ and $\alpha'$, and at a point close to $z_0$ with multiplicity $k$.
\end{itemize}
The constancy of $WC_0(X_{a,b},k,\alpha,\alpha',\bz,\bz')$ in the considered wall-crossing follows from the fact that the two local branches described in the former item intersect only in hyperbolic nodes or complex conjugate nodes independently of the mutual position of these branches, and from the transversality of the intersection of the above two families at $C$, which, in turn, follows from the cohomology vanishing (\ref{ertv18}).

{\it 7.3.} Suppose that $\xi\in\overline{\mathcal MC}^\R_{0,n}(X_{a,b},k,\alpha,\alpha',\vec\eps,\vec\eps')$ is a generic element of the family described in item (iii). First, we contract the component $\widehat C_2\subset\widehat C$ allowing the marked points $u_{ij},u_{i'j'}$ to coincide. Then the constancy of $WC_0(X_{a,b},k,\alpha,\alpha',\bz,\bz')$ in the considered wall-crossing can be derived by the argument of the preceding paragraph, when we replace the first family by the germ at $C$ the closure of the family of curves $C'\in|{\mathcal O}_{X_{a,b}}(k\Delta)|$ whose intersection with a small neighborhood of $z_{i'j'}$ is one smooth branch that intersects $D_a$ at $z_{i'j'}$ with multiplicity $\alpha_{i'j'}$ and at another point close to $z_{i'j'}$ with multiplicity $\alpha_{ij}$. The required transversality condition reduces to the same relation (\ref{ertv18}).

{\it7.4.} Suppose that $\xi\in\overline{\mathcal MC}^\R_{0,n}(X_{a,b},k,\alpha,\alpha',\vec\eps,\vec\eps')$ is a generic element of the family described in item (iv), and let $z_{ij}\in D_a$ be a singular point of $C=\bn(\widehat C)$.
We first consider a local deformation of the singular point $(C,z_{ij})$ and confirm that along this deformation the parity of the number of elliptic nodes in a neighborhood of $z_{ij}$ is constant. Then we prove that the wall-crossing along the considered path realizes this deformation, and hence conserves the number $WC_0(X_{a,b},k,\alpha,\alpha',\bz,\bz')$.

Deformations of the germ $(C,z_{ij})$ are described in \cite[Section 2(3) and Lemma 3]{Sh17}. Namely, in suitable local coordinates, the germ $(C,z_{ij})$ is given by
\begin{equation}\sigma_{02}y^2+\sigma_{2s+1,0}x^{2s+1}+\text{h.o.t.}=0,\quad2s+1=\alpha_{ij},\ \sigma_{02},\sigma_{2s+1,0}\in\R^*.\label{ertv21}\end{equation}
The miniversal deformation $B(C,z_{ij};D_a)$ of the germ $(C,z_{ij})$ in the space of curves intersecting $D_a$ at $z_{ij}$ with multiplicity $2s+1$ is spanned by the coefficients
$$\sigma_{r,1},\quad 0\le r\le s.$$ The deformation we are interested in corresponds to the equigeneric stratum $\qquad$ $EG(C,z_{ij};D_a)\subset B(C,z_{ij};D_a)$ whose elements parameterize curves in a neighborhood of $z_{ij}$ having $s=\delta(C,z_{ij})$ nodes (in addition to the condition of intersection with $D_a$ at $z_{ij}$ with multiplicity $2s+1$). This is a smooth one-parameter deformation with the tangent line $B(C,z_{ij};D_a)\cap I_1$, where
$$I_1=\{g\in{\mathcal O}_{X_{a,b},z_{ij}}\ :\ \ord g\big|_{C,z_{ij}}\ge4s+1=2s+1+2\delta(C,z_{ij})\}.$$
Furthermore, nodal germs belonging to $EG(C,z_{ij};D_a)$ have only imaginary nodes and elliptic nodes, and the number of latter ones is $\delta(C,z_{ij})\mod2$.

Recall that the curve $C$ is immersed outside $z_{ij}$ and that the number of elliptic nodes that may appear in a real nodal equigeneric deformation of an immersed singular point $z\in \Sing(C)$ equals $e(\xi,z)$ modulo $2$.
Next, we verify that the linear system in $|{\mathcal O}_{X_{a,b}}(k\Delta)|$ formed by the curves $C'$ such that
the function germ defining $(C',z_{ij})$ belongs to $I_1$, intersects transversally with the family of curves $C''\in|{\mathcal O}_{X_{a,b}}(k\Delta)|$ having $N$ nodes in the big torus outside a neighborhood of $z_{ij}$ and crossing the toric divisors at $\bz\setminus\{z_{ij}\}$ and at $\bz'$ with multiplicities determined by $\alpha\setminus\{\alpha_{ij}\}$ and $\alpha'$, respectively, and at a point close to $z_0$ with multiplicity $k$. Such a transversality amounts to the $h^1$-vanishing condition which, in fact, coincides with (\ref{ertv18}).
Hence, there exists a smooth local path $(\bz(\tau),\bz'(\tau))$ crossing the considered wall and realizing the local deformation presented in the preceding paragraph. Thus, the number $WC_0(X_{a,b},k,\alpha,\alpha',\bz,\bz')$ remains constant.

{\it7.5.} Suppose that $\xi\in\overline{\mathcal MC}^\R_{0,n}(X_{a,b},k,\alpha,\alpha',\vec\eps,\vec\eps')$ is a generic element of the family described in item (v). The germs $(C_1,z_0)$ and $(C_2,z_0)$ are smooth and intersect the toric divisor $D_{out}$ at $z_0$ with multiplicities $k_1$ and $k_2$, respectively. We are interested in a deformation which turns the germ $(C_1\cup C_2,z_0)$ into an immersed cylinder. Such kind of deformations was considered in \cite[Section 2.3 and Lemma 4]{Sh17}, but here we use another approach, in the style on the preceding part 7.4.

\begin{figure}
\setlength{\unitlength}{1mm}
\begin{picture}(130,90)(0,0)
\thinlines
\put(5,10){\vector(0,1){30}}\put(5,10){\vector(1,0){55}}
\put(5,60){\vector(0,1){30}}\put(5,60){\vector(1,0){55}}
\put(75,60){\vector(0,1){30}}\put(75,60){\vector(1,0){55}}
\dashline{2}(5,20)(25,20)\dashline{2}(25,10)(25,20)
\dashline{2}(5,70)(25,70)\dashline{2}(25,60)(25,70)\dashline{2}(95,60)(95,70)


\thicklines
\put(5,10){\line(2,1){20}}\put(5,30){\line(2,-1){20}}
\put(25,20){\line(3,-1){30}}\put(5,10){\line(0,1){20}}
\put(5,80){\line(2,-1){20}}\put(25,70){\line(3,-1){30}}
\put(75,80){\line(2,-1){20}}\put(95,70){\line(3,-1){30}}
\put(75,70){\line(1,0){20}}\put(75,70){\line(5,-1){50}}\put(75,70){\line(0,1){10}}

\put(1,19){$1$}\put(1,29){$2$}\put(1,69){$1$}\put(1,79){$2$}\put(71,69){$1$}\put(71,79){$2$}
\put(23,6){$k_1$}\put(23,56){$k_1$}\put(93,56){$k_1$}
\put(54,6){$k$}\put(54,56){$k$}\put(124,56){$k$}
\put(25,50){(a)}\put(95,50){(b)}\put(25,0){(c)}

\end{picture}
\caption{Proof of Proposition \ref{lrtv3}, Part 7.5}\label{frtv1}
\end{figure}
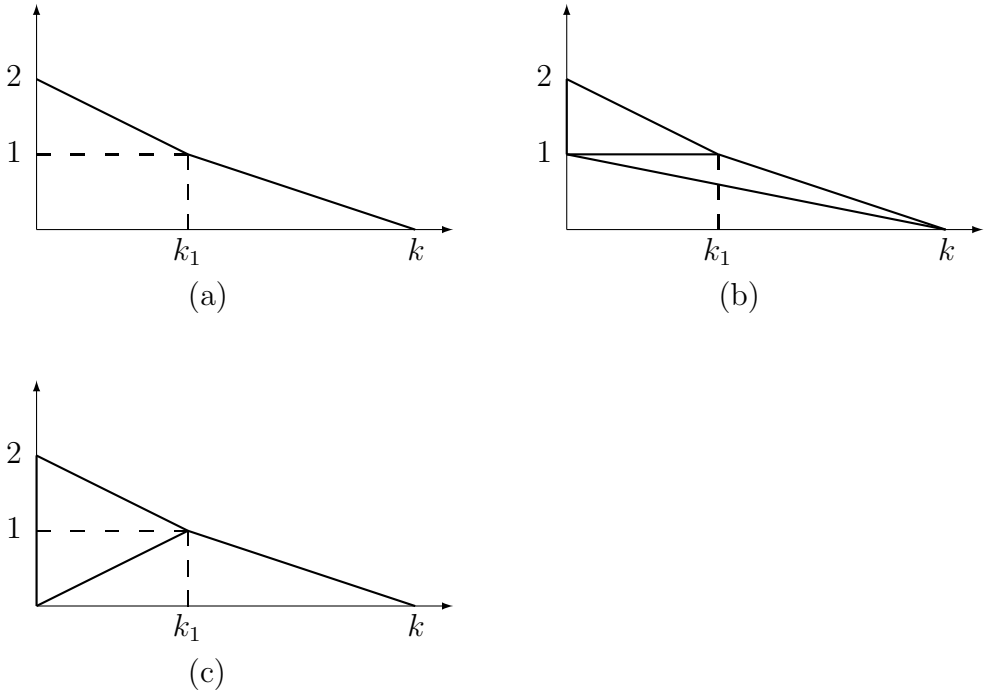

In suitable local coordinates, the Newton diagram at $z_0$ of a function defining the germ $(C_1\cup C_2,z_0)$ is as shown in Figure \ref{frtv1}(a). The miniversal deformation $B(C_1\cup C_2,z_0;D_{out})$ of the germ $(C_1\cup C_2,z_0)$ in the space of germs intersecting $D_{out}$ at $z_0$ with multiplicity $k$ is spanned by the coefficients
$$\sigma_{r,1},\quad 0\le r<k_1.$$
Furthermore, since the cylinder appearing in the deformation ${\mathcal D}$ must have $k_1-1$ nodes, the considered deformation ${\mathcal D}$ is, in fact, a one-parameter family. The tropicalization of ${\mathcal D}$ (understood as described in \cite[Section 2.5.4]{IMS}) defines a subdivision of the quadrangle with vertices $(0,1),(0,2),(k_1,1),(k,0)$, and the subdivision must be as shown in Figure \ref{frtv1}(b) with the truncation of the limit polynomial on the segment $[(0,1),(k_1,1)]$ having the form
\begin{equation}\sigma_{k_11}(x+\theta)^{k_1}y.\label{ertv19}\end{equation}
Indeed, otherwise the germ $(C_1\cup C_2,z_0)$ would turn into a surface with a handle when varying along ${\mathcal D}$. Formula (\ref{ertv19}) implies that the tangent line to ${\mathcal D}$ is given by
$$\sigma_{i0}=0,\quad 0\le i<k,\quad \sigma_{i1}=0,\quad 0\le i<k_1-1.$$ Note that these relations define the following ideal in ${\mathcal O}_{X_{a,b},z_0}$:
\begin{equation}I_2=\left\{\varphi\in{\mathcal O}_{X_{a,b},z_0}\ \Big|\ \begin{array}{c}
{\ord\varphi\big|_{(C_1,z_0)}\ge(C_1\cdot C_2)_{z_0}+k_1-1,}\\
{\ord\varphi\big|_{(C_2,z_0)}\ge(C_1\cdot C_2)_{z_0}+k_2-1}\end{array}\right\}\label{ertv20}\end{equation}
where $(C_1\cdot C_2)_{z_0}$ denotes the intersection number at the given point.
The higher terms depending on $\theta$ in the coefficients $\sigma_{i1}$, $0\le i<k_1$, can be found via the modification along the edge $[(0,1),(k_1,1)]$ (see \cite[Section 3.5]{Sh05} and \cite[Section 2.5.8]{IMS} \footnote{In \cite{Sh05,IMS} it is called ``refinement".}). It comes from the coordinate change $x=x'+\theta$ and leads to the new tropicalization associated with the subdivision shown in Figure \ref{frtv1}(c). The limit polynomial $F_{k_1}(x',y)$ with the Newton triangle $\conv\{(0,0),(0,2),(k_1,1)\}$ defines a rational curve, and its part restricted to the segment $[(0,1),(k_1,1)]$ is $P_{k_1}(x')y$, where $P_{k_1}$ is the $k_1$-th Chebyshev polynomial (up to a linear coordinate change and multiplication by a constant factor), see \cite[Proof of Proposition 6.1]{Sh05}. Furthermore, there are precisely $k_1$ polynomials $P_{k_1}$ matching the rationality condition and the fixed coefficients at $1$, $y^2$ and $(x')^{k_1}y$ (see \cite[Lemma 3.9]{Sh05}). To see this ramification analytically, we recall that (by the construction of the tropicalization) the exponents of $\theta$ in the coefficients of $F_{k_1}$ depend linearly on the exponents of $x'$ and $y$, whereas these exponents at the coefficients of $y^2$ and $(x')^{k_1}y$ vanish, and the constant term is $\sigma_{k,0}\theta^k$. Note also that the Chebyshev polynomial $P_{k_1}(x')$ contains only monomials $(x')^i$ with $k_1-i$ even. Thus, it follows that the minimal positive exponent of $\theta$ in the coefficients of $F_{k_1}$ is $\frac{k}{k_1}$. After the parameter change $\theta=(\theta')^{k_1}$, we obtain a regular parametrization of the family ${\mathcal D}$ by the parameter $\theta'$ containing $(\theta')^{k_1}$ as the minimal power, $(\theta')^k$ as the next power, and so on.

Over $\R$, we have the following picture. Recall that $k$ is assumed to be odd. If $k_1$ is odd, then the real part of ${\mathcal D}$ is a $C^1$-smooth germ. The elements of each of the components of ${\mathcal D}(\R)\setminus\{C\}$ correspond to the unique real Chebyshev polynomial $P_{k_1}(x')$ in the bunch of $k_1$ ones defined over $\C$, and the corresponding real rational curves $\{F_{k_1}=0\}$ have an even number of elliptic nodes (see \cite[Lemma 2.49(ii)]{IMS}).
If $k_1$ is even, then either there are no real Chebyshev polynomials $P_{k_1}(x')$, which means ${\mathcal D}(\R)=\{C\}$, or there are two real Chebyshev polynomials $P_{k_1}(x')$, which means that ${\mathcal D}(\R)$ is a real cuspidal curve such that the elements of one component of ${\mathcal D}(\R)\setminus\{C\}$ correspond to real rational curves $\{F_{k_1}=0\}$ with an odd number of elliptic nodes, while the elements of the other component of ${\mathcal D}(\R)\setminus\{C\}$ correspond to real rational curves $\{F_{k_1}=0\}$ without elliptic nodes (see \cite[Lemma 2.49(i)]{IMS}).

Note that the curve $C_1\cup C_2$ has
$$N=\frac{C_1^2+C_1K_{X_{a,b}}}{2}+1+\frac{C_2^2+C_2K_{X_{a,b}}}{2}+1+C_1C_2-k_1$$
nodes in the big torus. We claim that the following conditions imposed on the curves in the germ of $|{\mathcal O}_{X_{a,b}}(k\Delta)|$ at $C=C_1\cup C_2$ are transversal: (i) to have $N$ nodes in the big torus outside a neighborhood of $z_0$, (ii) to intersect the toric divisors $D_a$ and $D_b$ at $\bz$ and $\bz'$ with respective multiplicities determined by $\alpha$ and $\alpha'$, and (iii) to belong to the ideal $I_2$ defined by (\ref{ertv20}) at the point $z_0$. Indeed, this can be reformulated as the system of two $h^1$-vanishing relations:
$$H^1(\PP^1,{\mathcal O}_{\PP^1}(-\sum_{u_{ij}\in\widehat C_1}\alpha_{ij}u_{ij}-\sum_{u'_{ij}\in\widehat C_1}\alpha'_{ij}u'_{ij}-(k_1-1)u_{0,1})\otimes(\bn\big|_{\widehat C_1})^*{\mathcal O}_{X_{a,b}}(k_1\Delta))=0,$$
$$H^1(\PP^1,{\mathcal O}_{\PP^1}(-\sum_{u_{ij}\in\widehat C_2}\alpha_{ij}u_{ij}-\sum_{u'_{ij}\in\widehat C_2}\alpha'_{ij}u'_{ij}-(k_2-1)u_{0,2})\otimes(\bn\big|_{\widehat C_2})^*{\mathcal O}_{X_{a,b}}(k_2\Delta))=0,$$
where $u_{0,1}=\widehat C_1\cap\widehat C_3$, $u_{0,2}=\widehat C_2\cap\widehat C_3$. Both the relations hold since the degrees of the corresponding bundles are
$$(C_1^2+C_1K_{X_{a,b}}+1)-(C_1^2+C_1K_{X_{a,b}}+2)=-1>(2g-2)\big|_{g=0}=-2,$$
$$(C_2^2+C_2K_{X_{a,b}}+1)-(C_2^2+C_2K_{X_{a,b}}+2)=-1>(2g-2)\big|_{g=0}=-2,$$
respectively. The established transversality means that, in the wall-crossing
\begin{itemize}\item for an odd $k_1$, one of the counted real rational curves turns into another real rational curve having the same Welschinger sign,
\item for an even $k_1$, a pair of real rational curves having opposite Welschinger signs appear or disappear.
\end{itemize}

Thus, the number $WC_0(X_{a,b},k,\alpha,\alpha',\bz,\bz')$ remains constant in all wall-crossing events, which completes the proof of the proposition.
\end{proof}

\medskip
{\bf Acknowledgements.} The author is very grateful to Ilia Itenberg for stimulating discussions of the subject of the paper. I also thank Pierrick Bousseau for attracting my attention to the papers \cite{AB,Bou}. Special thanks are due to the unknown referee for numerous important remarks and corrections.

\end{document}